\documentclass{amsart}

\newcommand{\lvt}{\left|\kern-1.35pt\left|\kern-1.3pt\left|}
\newcommand{\rvt}{\right|\kern-1.3pt\right|\kern-1.35pt\right|}

\usepackage{tensor}
\usepackage[foot]{amsaddr}
\usepackage[english]{babel}
\usepackage[
pdftex,hyperfootnotes]{hyperref}
\usepackage[usenames,dvipsnames,svgnames,table,x11names,table]{xcolor}
\usepackage{pgfplots}
\usepackage{tikz}
\usepackage{tikz-3dplot}
\usetikzlibrary{calc,shadows,shapes.callouts,shapes.geometric,shapes.misc,positioning,patterns,decorations.pathmorphing,decorations.markings,decorations.fractals,decorations.pathreplacing,shadings,fadings}

\usepackage{longtable}
\usepackage{enumerate}
\allowdisplaybreaks

\usepackage{fancybox}
\usepackage[utf8]{inputenc}

\usepackage{tcolorbox}

\usepackage{rotating,multirow,pdflscape}
\usepackage{amssymb,latexsym,amsmath,amsthm}
\usepackage{mathrsfs}
\usepackage{mathtools,arydshln,mathdots}
\usepackage{bigints}
\usepackage{comment}

\usepackage{nicematrix}

\makeatletter
\renewcommand*\env@matrix[1][*\c@MaxMatrixCols c]{%
 \hskip -\arraycolsep
 \let\@ifnextchar\new@ifnextchar
 \array{#1}}
\makeatother

\mathtoolsset{centercolon}
\usepackage[toc,page]{appendix}

\usepackage{tocvsec2}

\usepackage{Baskervaldx}
\usepackage{newtxmath}

\newtheorem{teo}{Theorem}

\newtheorem{lemma}{Lemma}

\newtheorem{rem}{Remark}

\renewcommand{\d}{\operatorname{d}}

\renewcommand{\Re}{\operatorname{Re}}

\newcommand{\C}{\mathbb{C}}

\newcommand{\N}{\mathbb{N}}

\allowdisplaybreaks[1]

\makeatletter
\DeclareRobustCommand{\gaussk}{\DOTSB\gaussk@\slimits@}
\newcommand{\gaussk@}{\mathop{\vphantom{\sum}\mathpalette\bigcal@{K}}}

\newcommand{\bigcal@}[2]{%
	\vcenter{\m@th
		\sbox\z@{$#1\sum$}%
		\dimen@=\dimexpr\ht\z@+\dp\z@
		\hbox{\resizebox{!}{0.8\dimen@}{$\mathcal{K}$}}%
	}%
}
\newcommand{\cfracplus}{\mathbin{\cfracplus@}}
\newcommand{\cfracplus@}{%
	\sbox\z@{$\dfrac{1}{1}$}%
	\sbox\tw@{$+$}%
	\raisebox{\dimexpr\dp\tw@-\dp\z@\relax}{$+$}%
}
\newcommand{\cfracdots}{\mathord{\cfracdots@}}
\newcommand{\cfracdots@}{%
	\sbox\z@{$\dfrac{1}{1}$}%
	\sbox\tw@{$+$}%
	\raisebox{\dimexpr\dp\tw@-\dp\z@\relax}{$\cdots$}%
}
\makeatother

\makeatletter
\newcommand*{\relrelbarsep}{.386ex}
\newcommand*{\relrelbar}{%
	\mathrel{%
		\mathpalette\@relrelbar\relrelbarsep
	}%
}
\newcommand*{\@relrelbar}[2]{%
	\raise#2\hbox to 0pt{$\m@th#1\relbar$\hss}%
	\lower#2\hbox{$\m@th#1\relbar$}%
}
\providecommand*{\rightrightarrowsfill@}{%
	\arrowfill@\relrelbar\relrelbar\rightrightarrows
}
\providecommand*{\leftleftarrowsfill@}{%
	\arrowfill@\leftleftarrows\relrelbar\relrelbar
}
\providecommand*{\xrightrightarrows}[2][]{%
	\ext@arrow 0359\rightrightarrowsfill@{#1}{#2}%
}
\providecommand*{\xleftleftarrows}[2][]{%
	\ext@arrow 3095\leftleftarrowsfill@{#1}{#2}%
}
\makeatother

\newcommand*\pFqskip{8mu}
\catcode`,\active
\newcommand*\pFq{\begingroup
 \catcode`\,\active
 \def ,{\mskip\pFqskip\relax}%
 \dopFq
}
\catcode`\,12
\def\dopFq#1#2#3#4#5{%
 {}_{#1}F_{#2}\biggl[\genfrac..{0pt}{}{#3}{#4};#5\biggr]%
 \endgroup
}

\usepackage{xcolor} 

\usepackage[normalem]{ulem}
\usepackage{soul} 

\usepackage{comment} 

\usepackage{tikz}
\usetikzlibrary{shapes,arrows}
\usepackage{verbatim}

\tikzstyle{block} = [draw, rectangle, 
 minimum height=3em, minimum width=2em]

\usepackage{mathrsfs} 

\title[Hypergeometric Expressions for Type I Jacobi--Piñeiro Polynomials]
{Hypergeometric Expressions for
	Type I Jacobi--Piñeiro Orthogonal Polynomials with
	Arbitrary  Number of Weights}



\author[Branquinho]{Amílcar Branquinho$^1$}
\address{$^1$CMUC, Departamento de Matem\'atica,
Universidade de Coimbra, Largo D. Dinis, 3000-143 Coimbra, Portugal}
\email{$^1$ajplb@mat.uc.pt}

\author[Díaz]{Juan E. F. Díaz$^{2}$}
\address{$^2$CIDMA, Departamento de Matemática, Universidade de Aveiro, 3810-193 Aveiro, Portugal}
\email{$^2$juan.enri@ua.pt}

\author[Foulqué]{Ana Foulqui\'e-Moreno$^3$}
\address{$^3$CIDMA, Departamento de Matemática, Universidade de Aveiro, 3810-193 Aveiro, Portugal}
\email{$^3$foulquie@ua.pt}

\author[Mañas]{Manuel Mañas$^4$}
\address{$^4$Departamento de Física Teórica, Universidad Complutense de Madrid, 28040-Madrid, Spain}
\email{$^4$manuel.manas@ucm.es}

\thanks{$^1$Acknowledges Centro de Matemática da Universidade de Coimbra (CMUC) -- UIDB/00324/2020, funded by the Portuguese Government through FCT/MEC and co-funded by the European Regional Development Fund through the Partnership Agreement PT2020.}

\thanks{$^2$ and $^3$Acknowledges CIDMA Center for Research and Development in Mathematics and Applications (University of Aveiro) and the Portuguese Foundation for Science and Technology (FCT) within project
UIDB/04106/2020 and UIDP/04106/2020.}

\thanks{The authors acknowledges ``Agencia Estatal de Investigación'' research project [PID2021- 122154NB-I00], \emph{Ortogonalidad y aproximación con aplicaciones en machine learning y teoría de la probabilidad}.}

\begin{document}

\maketitle



\begin{abstract}
For a general number $p\geq 2$ of measures, we provide explicit expressions for the Jacobi--Piñeiro and Laguerre of the first kind multiple orthogonal polynomials of type I, presented in terms of multiple hypergeometric functions.
\end{abstract}

%

\section{Introduction}

	
Multiple orthogonal polynomials are a specialized family of polynomials that arise in various branches of mathematics and engineering. Unlike classical orthogonal polynomials, which are associated with a single weight function, multiple orthogonal polynomials are linked to multiple weight functions and measures simultaneously. These polynomials play a crucial role in diverse fields, including numerical analysis, approximation theory, and mathematical physics, offering powerful tools for solving complex problems involving simultaneous orthogonalities.
In the book \cite{Ismail} there is a comprehensive introduction to the subject and \cite{afm} for the relation with integrable systems.

Recently, the importance of multiple orthogonal polynomials in the Favard spectral description of banded bounded semi-infinite matrices has come to light. This connection has been explored in various sources such as \cite{aim, phys-scrip, Contemporary}, see also \cite{laa}. Additionally, their significance in the context of Markov chains and random walks beyond birth and death has been established, as seen in \cite{JP, hypergeometric, finite, CRM}. Notably, in both of these scenarios, type I polynomials play a pivotal role. Unfortunately, explicit expressions for type I multiple orthogonal polynomials are not widely available. In contrast, for type II polynomials, the Rodrigues' formula leads to explicit hypergeometric expressions that can accommodate an arbitrary number of `classical' weights, as detailed in \cite{AskeyII} and \cite{ContinuosII,Arvesu,Clasicos}. 

In our work presented in \cite{HahnI}, we offer explicit hypergeometric expressions for all type I polynomials within the realm of multiple Hahn orthogonal polynomials under the Askey scheme  \cite{AskeyII,askeyescalar}, a context encompassing only two weights.  Through these efforts, we have successfully constructed bidiagonal factorizations and characterized positive ones \cite{factorization}, as elaborated in \cite{aim, phys-scrip}.

In this paper, we provide explicit hypergeometric expressions for the type I  Jacobi--Piñeiro and type I Laguerre multiple orthogonal polynomials of the first kind, for the first time. To achieve this, we relied on tools acquired from our previous work in \cite{HahnI} and offered a fresh interpretation of results presented by Karp and Prilepkina in \cite{KP}.

For the upcoming discussion, it is essential to establish some fundamental concepts about multiple orthogonal polynomials \cite{nikishin_sorokin}. We begin by examining a system characterized by $p\geq 2$ weight functions, denoted as $w_1, \ldots, w_p$, which map from $\Delta\subseteq\mathbb R$ to $\mathbb R$, in conjunction with a measure $\mu:\Delta\subseteq\mathbb R\rightarrow\mathbb R^+$.

In certain cases, a sequence of type II polynomials $B_{\vec{n}}$ may exist, which are monic with $\deg{B}\leq|\vec{n}|$, satisfying the orthogonality relations 
\begin{align}
 \label{ortogonalidadTipoIIContinua}
 	\int_{\Delta}^{}x^j B_{\vec{n}}(x)w_i(x)\d\mu(x) =0,  && j \in\{0,\dots,n_i-1\}, && i \in \{1,\dots,p\} .
\end{align}
Additionally, there are $p$ sequences of type I polynomials, denoted as $A^{(1)}_{\vec{n}}, \ldots, A^{(p)}_{\vec{n}}$, with $\deg A^{(i)}_{\vec{n}}\leq n_i-1$, meeting the orthogonality conditions 
\begin{align}
\label{ortogonalidadTipoIContinua}
 \sum_{i=1}^p\int_{\Delta}^{}x^j A^{(i)}_{\vec{n}}(x)w_i(x)\d\mu(x)=
 \begin{cases}
 0,\;\text{if}\; j\in\{0,\ldots,|\vec{n}|-2\},\\
 1,\;\text{if}\; j=|\vec{n}|-1.
 \end{cases}
\end{align}
In this context, $\vec{n}=(n_1,\ldots,n_p)\in\mathbb N^p_0$ and $|\vec{n}|
\coloneqq
n_1+\cdots+n_p$. In this paper we will use the notation $\N\coloneq \{1,2,3,\dots\}$ and $\N_0\coloneq \{0\}\cup \N$.

Before delving into the main results, it's important to recall that, in the context of standard orthogonality and for certain families when $p=2$, the polynomials can be represented using generalized hypergeometric series, as discussed in \cite{andrews}. These series, denoted as
\begin{align}
\label{GeneralizedHypergeometricFuntions}
 \pFq{p}{q}{a_1,\ldots,a_p}{\alpha_1,\ldots,\alpha_q}{x}
  \coloneqq
 \sum_{l=0}^{\infty}\dfrac{(a_1)_l\cdots(a_p)_l}{(\alpha_1)_l\cdots(\alpha_q)_l}\dfrac{x^l}{l!}.
\end{align}
Here, $(x)_n$, $x\in\C$ and $n\in\N_0$, known as the Pochhammer symbol, is defined as:
\begin{align*}
 (x)_n
 \coloneqq
 \dfrac{\Gamma(x+n)}{\Gamma(x)}=\begin{cases}
 x(x+1)\cdots(x+n-1) , \;\text{if}\;n\in\N,\\
 1 , \;\text{if}\;n=0.
 \end{cases}
\end{align*}

To establish the results related to the type I polynomials, we will rely on the theorem by Karp and Prilepkina, as presented in \cite[Theorem~2.2]{KP}.
\begin{teo}
 Let be $a\in\mathbb C,\,\vec{f}=(f_1,\ldots,f_r)\in\mathbb C^r,\,\vec{b}=(b_1,\ldots,b_l)\in\mathbb C^l$,
 $\vec{m} =(m_1,\ldots, $ $m_r)\in\mathbb N^r$ and $\vec{k}=(k_1,\ldots,k_l)\in\mathbb N^l$ such that all the elements of the vector
 \begin{multline*}
 (\beta_1,\beta_2,\ldots,\beta_{k_1},\ldots,\beta_{k_1+\cdots+k_{l-1}+1},\beta_{k_1+\cdots+k_{l-1}+2},\ldots,\beta_{k_1+\cdots+k_l}) \\
 \coloneqq
(b_1,b_1+1,\ldots,b_1+k_1-1,\ldots,b_l,b_l+1,\ldots,b_l+k_l-1)
 \end{multline*}
 are distinct. If $\Re(k_1+\cdots+k_l-a-m_1-\cdots-m_r)>0$, then
 \begin{multline}
 \label{KPTheorem}
 \pFq{r+l+1}{r+l}{a,f_1+m_1,\ldots,f_r+m_r,b_1,\ldots,b_l}{f_1,\ldots,f_r,b_1+k_1,\ldots,b_l+k_l}{1}
 \\ =
\Gamma(1-a)\dfrac{(b_1)_{k_1}\cdots(b_l)_{k_l}}{(f_1)_{m_1}\cdots(f_r)_{m_r}} 
 \sum_{i=1}^{k_1+\cdots+k_l}\dfrac{\Gamma(\beta_i)}{\Gamma(\beta_i-a+1)}\dfrac{
 	\prod_{j=1}^{r} 	(f_j-\beta_i)_{m_j}
 }{	\tensor*{\prod\nolimits}{^{k_1+\cdots+k_l}_{j=1}^{,(i)}}^{}(\beta_j-\beta_i)}.
 \end{multline}
\end{teo}

\begin{rem}
Notice that here we have introduced the following notation
\begin{align*}
    {\prod}^{p,(i)}_{j=1}a_j, && i\in\{1,\ldots,p\} ,
\end{align*}
to indicate that the $i$-th factor is removed from the product.
In the forthcoming discussion, we will make extensive use of this as well as the following notation:
\begin{align*}
	&\vec{e}_p
\coloneqq
(1,\ldots,1)\in\mathbb R^p .
\end{align*}
Given a vector $\vec{a}\in\mathbb R^p$, the notation $\vec{a}^{(i)}\in\mathbb R^{p-1}$ indicates that the $i$-th component of $\vec{a}$ has been removed.
\end{rem}

However, in this context, we will rephrase this theorem into an equivalent, more practical form
\begin{lemma}
\label{KP}
 Let be $a\in\mathbb C,\,\vec{f}=(f_1,\ldots,f_r)\in\mathbb C^r,\,\vec{b}=(b_1,\ldots,b_l)\in\mathbb C^l$ with all its components distinct, $\vec{m}=(m_1,\ldots,m_r)\in\mathbb N^r$ and $\vec{k}=(k_1,\ldots,k_l)\in\mathbb N^l$. If $\Re(k_1+\cdots+k_l-a-m_1-\cdots-m_r)>0$, then
 \begin{multline}
\label{KPTheoremReformulated}
 \pFq{r+l+1}{r+l}{a,f_1+m_1,\ldots,f_r+m_r,b_1,\ldots,b_l}{f_1,\ldots,f_r,b_1+k_1,\ldots,b_l+k_l}{1}\\=\Gamma(1-a)\dfrac{(b_1)_{k_1}\cdots(b_l)_{k_l}}{(f_1)_{m_1}\cdots(f_r)_{m_r}} \sum_{q=1}^{l}\dfrac{(-1)^{k_q-1}
 	 	\prod_{j=1}^{r} 	(\tilde f_j-m_j)_{m_j}
 \Gamma(b_q+k_q-1)}{(k_q-1)!\prod^{l,(q)}_{j=1}(b_j-b_q-k_q+1)_{k_j}\Gamma(b_q+k_q-a)} \\
 \times\pFq{r+l+1}{r+l}{-k_q+1,-b_q-k_q+1+a, \tilde f_1
 ,\ldots, \tilde f_r
 ,\vec b^{(q)}-(b_q+k_q-1) \vec{e}_{l-1}
 }{-b_q-k_q+2,
 \tilde f_1-m_1
 ,\ldots,
 \tilde f_r-m_r
 ,\,\vec{b}^{(q)}+\vec{k}^{(q)}-(b_q+k_q-1)\vec{e}_{l-1}}{1} ,
\end{multline}
with $\tilde f_j = f_j-b_q+1-k_q+m_j$, $j = 1, \ldots , r$.
\end{lemma}

\begin{proof}
	For the sake of clarity, we'll temporarily assign a new name to the expression within the sum in Equation \eqref{KPTheorem} as follows:
\begin{align*}
 F(\beta_i)
 \coloneqq
 \dfrac{\Gamma(\beta_i)}{\Gamma(\beta_i-a+1)}\dfrac{ 	\prod_{j=1}^{r} 	(f_j-\beta_i)_{m_j}}{\prod_{j=1}^{k_1+\cdots+k_l,(i)}(\beta_j-\beta_i)}.
\end{align*}
	Now, we can reorganize this sum as follows:
\begin{align*}
 \sum_{i=1}^{k_1+\cdots+k_l}F(\beta_i)
 & =\sum_{i=1}^{k_1}F(\beta_i)+\sum_{i=k_1+1}^{k_1+k_2}F(\beta_i)+\cdots+\sum_{i=k_1+\cdots+k_{l-1}+1}^{k_1+\cdots+k_l}F(\beta_i)
 \\
  &=\sum_{i=1}^{k_1}F(b_1+i-1)+\sum_{i=1}^{k_2}F(b_2+i-1)+\cdots+\sum_{i=1}^{k_l}F(b_l+i-1) \\
  &=\sum_{q=1}^{l}\sum_{i=1}^{k_q}F(b_q+i-1) .
\end{align*}
	Replacing $F(b_q+i-1)$ and simplifying, we can rewrite each $i$-labeled sum as a $_{r+l+1}F_{r+l}$ function, resulting in the following expression:
\begin{multline*}
 \pFq{r+l+1}{r+l}{a,f_1+m_1,\ldots,f_r+m_r,b_1,\ldots,b_l}{f_1,\ldots,f_r,b_1+k_1,\ldots,b_l+k_l}{1}=\Gamma(1-a)\dfrac{(b_1)_{k_1}\cdots(b_l)_{k_l}}{(f_1)_{m_1}\cdots(f_r)_{m_r}}\\
 \times\sum_{q=1}^{l}\dfrac{\prod_{j=1}^r(f_j-b_q)_{m_j}\Gamma(b_q)}{(k_q-1)!\prod_{j=1}^{l,(q)}(b_j-b_q)_{k_j}\Gamma(b_q+1-a)}\\
 \times\pFq{r+l+1}{r+l}{-k_q+1,b_q,-f_1+b_q+1,\ldots,-f_r+b_q+1,
 \,(b_q+1)\vec{e}_{l-1}-\vec{b}^{(q)}-\vec{k}^{(q)}}{b_q+1-a,-f_1+b_q-m_1+1,\ldots,-f_r+b_q-m_r+1,
 \,(b_q+1)\vec{e}_{l-1}-\vec{b}^{(q)}}{1}.
\end{multline*}
	Finally, applying the following reversal formula, as seen in \cite[Equation 2.2.3.2]{slater}:
\begin{multline*}
 \pFq{p+1}{q}{-n,a_1,\ldots,a_p}{b_1,\ldots,b_q}{1}
 \\=
(-1)^n\dfrac{(a_1)_n\cdots(a_p)_n}{(b_1)_n\cdots(b_q)_{n}}\pFq{q+1}{p}{-n,-b_1-n+1,\ldots,-b_q-n+1}{-a_1-n+1,\ldots,-a_p-n+1}{1}.
\end{multline*}
	We can apply this reversal formula to each of the $_{r+l+1}F_{r+l}$ functions within the $q$-labeled sum and simplify to obtain expression \eqref{KPTheoremReformulated}.
\end{proof}

\section{Hypergeometric expressions for type I Jacobi--Piñeiro multiple orthogonal polynomials}

For this family, the weight functions are defined as follows:
\begin{align}
 \label{WeightsJP}
 	w_{i}(x;\alpha_i) =x^{\alpha_i},  && i \in\{1,\ldots,p\}, && \d\mu(x) =(1-x)^\beta\d x, && \Delta =[0,1] .
\end{align}
Here $\vec{\alpha}=(\alpha_1,\ldots,\alpha_p)$ with $\alpha_1,\ldots,\alpha_p,\beta>-1$, and, to establish an AT system, we require that $\alpha_i-\alpha_j\not\in\mathbb Z$ for $i\neq j$.
It's also important to note that the moments are given by:
\begin{align}
\label{MomentJP}
\int_{0}^{1}x^{\alpha_i+k}(1-x)^\beta\d x
=\dfrac{\Gamma(\beta+1)\Gamma(\alpha_i+k+1)}{\Gamma(\alpha_i+\beta+k+2)},
&& k \in \mathbb N_0 .
\end{align}

In \cite[\S 3]{HahnI} and \cite{JP}, the type I polynomials were derived for the case when $p=2$ and expressed using a $_3F_2$ function as follows:
\begin{multline*}
 P^{(i)}_{(n_1,n_2)}(x;\alpha_1,\alpha_2,\beta)
 =(-1)^{n_1+n_2-1}\dfrac{(\alpha_1+\beta+n_1+n_2)_{n_1}({\alpha}_2+\beta+n_1+n_2)_{{n}_2}}{(n_i-1)!(\hat{\alpha}_i-\alpha_i)_{\hat{n}_i}}\\
\times\dfrac{\Gamma(\alpha_i+\beta+n_1+n_2)}{\Gamma(\beta+n_1+n_2)\Gamma(\alpha_i+1)}\pFq{3}{2}{-n_i+1,\alpha_i+\beta+n_1+n_2,\alpha_i-\hat{\alpha}_i-\hat{n}_i+1}{\alpha_i+1,\alpha_i-\hat{\alpha}_i+1}{x} .
\end{multline*}
Here, $\hat{\alpha}_i
\coloneqq
\delta_{i,2}\alpha_1+\delta_{i,1}\alpha_2$, $\hat{n}_i
\coloneqq
\delta_{i,2}n_1+\delta_{i,1}n_2$ with $i\in\{1,2\}$.

In an attempt to generalize this expression, we arrive at the following theorem:
\begin{teo}
\label{JPTypeITheorem}
The Jacobi--Piñeiro polynomials of type I for $p\geq2$ are given by:
\begin{multline}
\label{JPTypeI}
 P^{(i)}_{\vec{n}} (x;\alpha_1,\ldots,\alpha_p,\beta)
  =(-1)^{|\vec{n}|-1}\dfrac{\prod_{k=1}^{p}(\alpha_k+\beta+|\vec{n}|)_{n_k}}{(n_i-1)!\prod_{k= 1}^{p,(i)}(\alpha_k-\alpha_i)_{n_k}}
  \\
  \times\dfrac{\Gamma(\alpha_i+\beta+|\vec{n}|)}{\Gamma(\beta+|\vec{n}|)\Gamma(\alpha_i+1)}
  \pFq{p+1}{p}{-n_i+1,\alpha_i+\beta+|\vec{n}|, (\alpha_i+1)\vec{e}_{p-1}-\vec{\alpha}^{(i)}-\vec{n}^{(i)}}{\alpha_i+1,(\alpha_i+1)\vec{e}_{p-1}-\vec{\alpha}^{(i)}}{x} .
\end{multline} 
\end{teo}
\begin{proof}
Initially, we will leverage Equation \eqref{MomentJP} and our understanding of hypergeometric functions to derive a more convenient expression for the following integral:
\begin{multline*}
\int_{0}^{1} x^j P^{(i)}_{\vec{n}}(x;\alpha_1,\ldots,\alpha_p,\beta)x^{\alpha_i}(1-x)^{\beta}\d x\\
\begin{aligned}
	&=\begin{multlined}[t][.8\textwidth]
	(-1)^{|\vec{n}|-1}\dfrac{\prod_{k=1}^{p}(\alpha_k+\beta+|\vec{n}|)_{n_k}}{(n_i-1)!\prod_{k= 1}^{p,(i)}(\alpha_k-\alpha_i)_{n_k}} 
 \\
 \times\dfrac{\Gamma(\alpha_i+\beta+|\vec{n}|)}{\Gamma(\beta+|\vec{n}|)\Gamma(\alpha_i+1)}
\sum_{l=0}^{n_i-1}\dfrac{(-n_i+1)_l}{l!}\dfrac{(\alpha_i+\beta+|\vec{n}|)_l}{(\alpha_i+1)_l}\prod_{k= 1}^{p,(i)} \dfrac{(\alpha_i-\alpha_k-n_k+1)_l}{(\alpha_i-\alpha_k+1)_l} \\
 \times\underbrace{\int_{0}^{1}x^{\alpha_i+j+l}(1-x)^\beta\d x}_{\frac{\Gamma(\beta+1)\Gamma(\alpha_i+j+l+1)}{\Gamma(\alpha_i+\beta+j+l+2)}\,\text{by \eqref{MomentJP}}}
\end{multlined}\\
 &=\begin{multlined}[t][.8\textwidth]
 	(-1)^{|\vec{n}|-1}\dfrac{\prod_{k=1}^{p}(\alpha_k+\beta+|\vec{n}|)_{n_k}}{(n_i-1)!\prod_{k= 1}^{p,(i)}(\alpha_k-\alpha_i)_{n_k}} \dfrac{\Gamma(\alpha_i+\beta+|\vec{n}|)(\alpha_i+1)_j}{\Gamma(\alpha_i+\beta+j+2)(\beta+1)_{|\vec{n}|-1}}\\
 \times\pFq{p+2}{p+1}{-n_i+1,\alpha_i+j+1,\alpha_i+\beta+|\vec{n}|,(\alpha_i+1)\vec{e}_{p-1}-\vec{\alpha}^{(i)}-\vec{n}^{(i)}}{\alpha_i+1,\alpha_i+\beta+j+2,(\alpha_i+1)\vec{e}_{p-1}-\vec{\alpha}^{(i)}}{1}.
 \end{multlined}
\end{aligned}
\end{multline*}
In this scenario, the orthogonality conditions \eqref{ortogonalidadTipoIContinua} that we need to establish are equivalent to the following hypergeometric summation formula.
\begin{multline}
\label{OrthogonalityJPTypeI}
 (-1)^{|\vec{n}|-1} \dfrac{\prod_{k=1}^{p}(\alpha_k+\beta+|\vec{n}|)_{n_k}}{(\beta+1)_{|\vec{n}|-1}}
 \sum_{i=1}^{p}\dfrac{\Gamma(\alpha_i+\beta+|\vec{n}|)(\alpha_i+1)_j}{(n_i-1)!\prod_{k= 1}^{p,(i)}(\alpha_k-\alpha_i)_{n_k}\Gamma(\alpha_i+\beta+j+2)}\\
 \times\pFq{p+2}{p+1}{-n_i+1,\alpha_i+j+1,\alpha_i+\beta+|\vec{n}|,(\alpha_i+1)\vec{e}_{p-1}-\vec{\alpha}^{(i)}-\vec{n}^{(i)}}{\alpha_i+1,\alpha_i+\beta+j+2,(\alpha_i+1)\vec{e}_{p-1}-\vec{\alpha}^{(i)}}{1}
 \\
 =\begin{cases}
 0,\;\text{if}\; j\in\{0,\ldots,|\vec{n}|-2\},\\
 1,\;\text{if}\; j=|\vec{n}|-1.
 \end{cases}
\end{multline}
We will now apply Equation \eqref{KPTheoremReformulated} to the preceding $_{p+2}F_{p+1}$ hypergeometric function.

For $j\in\{0,\ldots,|n|-2\}$, in accordance with the notation introduced in Lemma \ref{KP}, we can identify
\begin{align*}
 &a=-n_i+1; \quad (f_1,f_2)=(\alpha_i+1,\alpha_i+\beta+j+2), \quad (m_1,m_2)=(j,|\vec{n}|-2-j);\\
 &(b_1,\ldots,b_{p-1})=(\alpha_i+1)\vec{e}_{p-1}-\vec{\alpha}^{(i)}-\vec{n}^{(i)},\quad (k_1,\ldots,k_{p-1})=\vec{n}^{(i)} .
\end{align*}
The condition $\Re(k_1+\cdots+k_{p-1}-a-m_1-m_2)=\Re(|\vec{n}|-n_i+n_i-1-j-|\vec{n}|+2+j)=1>0$ is met, allowing us to apply Equation \eqref{KPTheoremReformulated}. After some calculations, we can express the $_{p+2}F_{p+1}$ hypergeometric series in \eqref{OrthogonalityJPTypeI} as follows:
\begin{multline*}
 \dfrac{\Gamma(\alpha_i+\beta+|\vec{n}|)(\alpha_i+1)_j}{(n_i-1)!\prod_{k= 1}^{p,(i)}(\alpha_k-\alpha_i)_{n_k}\Gamma(\alpha_i+\beta+j+2)}
 \\
\times
\pFq{p+2}{p+1}{-n_i+1,\alpha_i+j+1,\alpha_i+\beta+|\vec{n}|,(\alpha_i+1)\vec{e}_{p-1}-\vec{\alpha}^{(i)}-\vec{n}^{(i)}}{\alpha_i+1,\alpha_i+\beta+j+2,(\alpha_i+1)\vec{e}_{p-1}-\vec{\alpha}^{(i)}}{1}
 \\=
 -\sum_{ q=1,q\neq i}^{p}\dfrac{\Gamma(\alpha_q+\beta+|\vec{n}|)(\alpha_q+1)_j}{(n_q-1)!\prod_{k= 1}^{p,(q)}(\alpha_k-\alpha_q)_{n_k}\Gamma(\alpha_q+\beta+j+2)}
  \\
 \times
 \pFq{p+2}{p+1}{-n_q+1,\alpha_q+j+1,\alpha_q+\beta+|\vec{n}|,(\alpha_q+1)\vec{e}_{p-1}-\vec{\alpha}^{(q)}-\vec{n}^{(q)}}{\alpha_q+1,\alpha_q+\beta+j+2,(\alpha_q+1)\vec{e}_{p-1}-\vec{\alpha}^{(q)}}{1} .
\end{multline*}
The left-hand side of Equation \eqref{OrthogonalityJPTypeI} equals 0 for $j=0,\ldots,|\vec{n}|-2$.

Now, when $j=|\vec{n}|-1$, referring to the notation introduced in Lemma \ref{KP}, we can identify:
\begin{align*}
 &a=-n_i+1; \quad f=\alpha_i+1, \quad m=|\vec{n}|-1,\\
 &(b_1,\ldots,b_{p})=\left(\alpha_i+\beta+|\vec{n}|,\,(\alpha_i+1)\vec{e}_{p-1}-\vec{\alpha}^{(i)}-\vec{n}^{(i)}\right) ,
  \\
 &(k_1,\ldots,k_{p})=\left(1,\,\vec{n}^{(i)}\right) .
\end{align*}
The condition $\Re(k_1+\cdots+k_{p}-a-m)=\Re(1+|\vec{n}|-n_i+n_i-1-|\vec{n}|+1)=1>0$ is met, allowing us to apply Equation \eqref{KPTheoremReformulated}. This time, we introduce an additional parameter~$b$, which results in an extra term in the decomposition. After some calculations, we can express the $_{p+2}F_{p+1}$ function in \eqref{OrthogonalityJPTypeI} as follows:
\begin{multline*}
 \dfrac{\Gamma(\alpha_i+\beta+|\vec{n}|)(\alpha_i+1)_{|\vec{n}|-1}}{(n_i-1)!\prod_{k= 1}^{p,(i)}(\alpha_k-\alpha_i)_{n_k}\Gamma(\alpha_i+\beta+|\vec{n}|+1)}
 \\
 \times
 \pFq{p+2}{p+1}{-n_i+1,\alpha_i+|\vec{n}|,\alpha_i+\beta+|\vec{n}|,(\alpha_i+1)\vec{e}_{p-1}-\vec{\alpha}^{(i)}-\vec{n}^{(i)}}{\alpha_i+1,\alpha_i+\beta+|\vec{n}|+1,(\alpha_i+1)\vec{e}_{p-1}-\vec{\alpha}^{(i)}}{1}
 \\
=(-1)^{|\vec{n}|-1}\dfrac{(\beta+1)_{|\vec{n}|-1}}{\prod_{k=1}^{p}(\alpha_k+\beta+|\vec{n}|)_{n_k}}
 -\sum_{ q=1,q\neq i}^{p}\dfrac{\Gamma(\alpha_q+\beta+|\vec{n}|)(\alpha_q+1)_{|\vec{n}|-1}}{(n_q-1)!\prod_{k= 1}^{p,(q)}(\alpha_k-\alpha_q)_{n_k}\Gamma(\alpha_q+\beta+|\vec{n}|+1)}\\
 \times
 \pFq{p+2}{p+1}{-n_q+1,\alpha_q+|\vec{n}|,\alpha_q+\beta+|\vec{n}|,(\alpha_q+1)\vec{e}_{p-1}-\vec{\alpha}^{(q)}-\vec{n}^{(q)}}{\alpha_q+1,\alpha_q+\beta+|\vec{n}|+1,(\alpha_q+1)\vec{e}_{p-1}-\vec{\alpha}^{(q)}}{1}.
\end{multline*}
For $j=|\vec{n}|-1$, the left-hand side of Equation \eqref{OrthogonalityJPTypeI} equals $1$.
\end{proof}

\section{Hypergeometric expressions for type I Laguerre of first kind multiple orthogonal polynomials}
In this case, we define the weight functions as:
\begin{align}
 \label{WeightsLaguerreI}
 w_{i}(x;\alpha_i)=e^{-x}x^{\alpha_i}; && i=1,\ldots,p; && \d\mu(x)=\d x; && \Delta=[0,\infty) .
\end{align}
Here $\vec{\alpha}=(\alpha_1,\ldots,\alpha_p)$ with $\alpha_1,\ldots,\alpha_p>-1$, and, to establish an AT system, we require that $\alpha_i-\alpha_j\not\in\mathbb Z$ for $i\neq j$.

In \cite[\S 7]{HahnI}, it was demonstrated that the type I polynomials for $p=2$ could be derived from the Jacobi--Piñeiro polynomials through the following limit:
\begin{multline*}
 L^{(i)}_{(n_1,n_2)}(x;\alpha_1,\alpha_2)\\=\lim_{\beta\rightarrow\infty}
 \dfrac{\Gamma(\beta+n_1+n_2)}{(\alpha_1+\beta+n_1+n_2)_{n_1}(\alpha_2+\beta+n_1+n_2)_{n_2}\Gamma(\alpha_i+\beta+n_1+n_2)}
 P_{(n_1,n_2)}^{(i)}\left(\dfrac{x}{\beta};\alpha_1,\alpha_2,\beta\right).
\end{multline*}
This yields the expression:
\begin{multline*}
 L_{(n_1,n_2)}^{(i)}(x;\alpha_1,\alpha_2)
 \\ 
 =(-1)^{n_1+n_2-1}\dfrac{1}{(n_i-1)!\Gamma(\alpha_i+1)(\hat{\alpha}_i-\alpha_i)_{\hat{n}_i}}
 \pFq{2}{2}{-n_i+1,\alpha_i-\hat{\alpha}_i-\hat{n}_i+1}{\alpha_i+1,\alpha_i-\hat{\alpha}_i+1}{x} .
\end{multline*}
Here, $\hat{\alpha}_i=\delta_{i,2}\alpha_1+\delta_{i,1}\alpha_2$, and $\hat{n}_i=\delta_{i,2}n_1+\delta_{i,1}n_2$ with  $i\in\{1,2\}$.

Once again, in an attempt to generalize this expression, we arrive at:
\begin{teo}
\label{LaguerreITypeITheorem}
The Laguerre of the first kind polynomials of type I for $p\geq2$ are expressed as:
\begin{multline}
\label{LaguerreITypeI}
 L^{(i)}_{\vec{n}}(x;\alpha_1,\ldots,\alpha_p)
 =(-1)^{|\vec{n}|-1}\dfrac{1}{(n_i-1)!\prod_{i\neq k=1}^{p}(\alpha_k-\alpha_i)_{n_k}\Gamma(\alpha_i+1)}
  \\
 \times \pFq{p}{p}{-n_i+1, (\alpha_i+1)\vec{e}_{p-1}-\vec{\alpha}^{(i)}-\vec{n}^{(i)}}{\alpha_i+1,(\alpha_i+1)\vec{e}_{p-1}-\vec{\alpha}^{(i)}}{x} .
\end{multline}
These polynomials are derived by applying the limit to the Jacobi--Piñeiro type I polynomials \eqref{JPTypeI}, as described in:
\begin{multline}
\label{LaguerreITypeIasLimitJPTypeI}
 L^{(i)}_{\vec{n}}(x;\alpha_1,\ldots,\alpha_p)
 \\=
\lim_{\beta\rightarrow\infty}\dfrac{\Gamma(\beta+|\vec{n}|)}{\prod_{k=1}^{p}(\alpha_k+\beta+|\vec{n}|)_{n_k}\Gamma(\alpha_i+\beta+|\vec{n}|)}P_{\vec{n}}^{(i)}\left(\dfrac{x}{\beta};\alpha_1,\ldots,\alpha_p,\beta\right) .
\end{multline}
\end{teo}

\begin{proof}
In this context, we will leverage the fact that the polynomials defined in \eqref{LaguerreITypeI} can be obtained through Jacobi--Piñeiro as shown in \eqref{JPTypeI} via the limit \eqref{LaguerreITypeIasLimitJPTypeI}. In Theorem \ref{JPTypeITheorem}, we have established the following orthogonality conditions for the Jacobi--Piñeiro polynomials $P_{\vec{n}}^{(i)}$:
\begin{align*}
 \sum_{i=1}^p\int_{0}^{1}x^{j}P^{(i)}_{\vec{n}}(x;\alpha_1,\ldots,\alpha_p,\beta)x^{\alpha_i}(1-x)^\beta\d x=
 \begin{cases}
 0,\;\text{if}\;j\in \{0,\ldots,|\vec{n}|-2\},\\
 1,\;\text{if}\;j=|\vec{n}|-1.
 \end{cases}
\end{align*}
This expression can be reformulated by substituting the variable $x$ with $\frac{x}{\beta}$ and introducing certain constants as follows:
\begin{multline*}
 \sum_{i=1}^p \dfrac{1}{\beta^{|\vec{n}|-1-j}}\dfrac{\prod_{k=1}^{p}(\alpha_k+\beta+|\vec{n}|)_{n_k}\Gamma(\alpha_i+\beta+|\vec{n}|)}{\Gamma(\beta+|\vec{n}|)\beta^{\alpha_i+j+1}}\\
 \times\int_{0}^{\beta}x^{j}\dfrac{\Gamma(\beta+|\vec{n}|)}{\prod_{k=1}^{p}(\alpha_k+\beta+|\vec{n}|)_{n_k}\Gamma(\alpha_i+\beta+|\vec{n}|)}P^{(i)}_{\vec{n}}\left(\dfrac{x}{\beta};\alpha_1,\ldots,\alpha_p,\beta\right)x^{\alpha_i}\left(1-\dfrac{x}{\beta}\right)^\beta\d x
  \\
 =
 \begin{cases}
 {0},\;\text{if}\;j\in\{0,\ldots,|\vec{n}|-2\},\\
 {1},\;\text{if}\;j=|\vec{n}|-1.\\
 \end{cases}
\end{multline*}
Now, our objective is to apply the limit as $\beta\rightarrow\infty$ to both sides of the preceding equation. To ensure that the limit can be interchanged with the integral, we need to establish that the modulus of the expression within the integral is bounded for all $\beta>0$.

On one hand, we can establish that:
\begin{align*}
	 \left(1-\dfrac{x}{\beta}\right)^\beta I_{[0,\beta)} \leq \mathrm e^{-x}, &&\beta >0, && x \in\mathbb R^+.
\end{align*}
On the other hand, based on its definition, the coefficients of the polynomial:
\begin{align*}
 \dfrac{\Gamma(\beta+|\vec{n}|)}{\prod_{k=1}^{p}(\alpha_k+\beta+|\vec{n}|)_{n_k}\Gamma(\alpha_i+\beta+|\vec{n}|)}P^{(i)}_{\vec{n}}\left(\dfrac{x}{\beta};\alpha_1,\ldots,\alpha_p,\beta\right)
\end{align*}
are bounded by the coefficients of the respective polynomial $L_{\vec{n}}^{(i)}(x;\alpha_1,\ldots,\alpha_p)$, cf. \eqref{LaguerreITypeI}, for all $\beta>0$ and $x\in\mathbb R^+$. This implies that the modulus of the aforementioned polynomial can be bounded for all $\beta>0$ and $x\in\mathbb R^+$ by the sum of the moduli of the coefficients of $L_{\vec{n}}^{(i)}(x;\alpha_1,\ldots,\alpha_p)$ plus an additive constant. This entire sum can then be multiplied by $x^{n_i-1}$ if $x>1$ or by $x^0$ if $x\leq1$.

So, we have shown that the expression within the integral is bounded for all $\beta>0$ and all $x\in\mathbb R^+$. By applying Lebesgue's dominated convergence theorem, we can interchange the limit as $\beta\rightarrow\infty$
\begin{multline*}
 \sum_{i=1}^p
 \overbrace{\lim_{\beta\rightarrow\infty}\dfrac{1}{\beta^{|\vec{n}|-1-j}}\dfrac{\prod_{k=1}^{p}(\alpha_k+\beta+|\vec{n}|)_{n_k}\Gamma(\alpha_i+\beta+|\vec{n}|)}{\Gamma(\beta+|\vec{n}|)\beta^{\alpha_i+j+1}}}^{1}\\
 \times\int_{0}^{\infty}x^{j}
 \underbrace{\lim_{\beta\rightarrow\infty}\dfrac{\Gamma(\beta+|\vec{n}|)}{\prod_{k=1}^{p}(\alpha_k+\beta+|\vec{n}|)_{n_k}\Gamma(\alpha_i+\beta+|\vec{n}|)}P^{(i)}_{\vec{n}}\left(\dfrac{x}{\beta};\alpha_1,\ldots,\alpha_p,\beta\right)}_{L^{(i)}_{\vec{n}}(x;\alpha_1,\ldots,\alpha_p)\,\text{by \eqref{LaguerreITypeIasLimitJPTypeI}}}
 x^{\alpha_i} \\
 \times\underbrace{\lim_{\beta\rightarrow\infty}\left(1-\dfrac{x}{\beta}\right)^\beta}_{\mathrm e^{-x}}\d x
 =\begin{cases}
 0,\;\text{if}\;j\in\{0,\ldots,|\vec{n}|-2\},\\
 1,\;\text{if}\;j=|\vec{n}|-1.
 \end{cases}
\end{multline*}
We have established the orthogonality conditions \eqref{ortogonalidadTipoIContinua} for the polynomials \eqref{LaguerreITypeI} with respect to the weight functions \eqref{WeightsLaguerreI}.
\end{proof}

\section{Conclusions and outlook}

In this paper, we have successfully derived explicit generalized hypergeometric expressions for two families of type I multiple orthogonal polynomials with an arbitrary number of weights. This finding is particularly pertinent in contexts where type I polynomials play a crucial role, such as their application in Markov chains with transition matrices featuring $p$ subdiagonals.

For future research, an essential objective is to identify similar hypergeometric expressions using Kampé de Fériet functions for Hahn multiple orthogonal polynomials with $p$ measures. This exploration within the Askey scheme will enable us to acquire corresponding hypergeometric expressions for the multiple orthogonal  polynomial of type I families under Hahn.

\end{document}